\documentclass[12pt]{amsart}
\usepackage{amsmath,amssymb,amsbsy,amsfonts,amsthm,latexsym,
                        amsopn,amstext,amsxtra,euscript,amscd,mathrsfs,color,bm}
                        
\usepackage{float} 
\usepackage[english]{babel}
\usepackage{url}
\usepackage[colorlinks,linkcolor=blue,anchorcolor=blue,citecolor=blue,backref=page]{hyperref}
\usepackage[norefs,nocites]{refcheck}

 \usepackage[np]{numprint}

\npdecimalsign{\ensuremath{.}}

\usepackage{bibentry}

\usepackage[english]{babel}
\usepackage{mathtools}
\usepackage{todonotes}
\usepackage{url}
\usepackage[colorlinks,linkcolor=blue,anchorcolor=blue,citecolor=blue,backref=page]{hyperref}

\begin{document}

\newcommand{\commA}[2][]{\todo[#1,color=yellow]{A: #2}}
\newcommand{\commI}[2][]{\todo[#1,color=green!60]{I: #2}}
    
\newtheorem{theorem}{Theorem}
\newtheorem{lemma}[theorem]{Lemma}
\newtheorem{example}[theorem]{Example}
\newtheorem{algol}{Algorithm}
\newtheorem{corollary}[theorem]{Corollary}
\newtheorem{prop}[theorem]{Proposition}
\newtheorem{definition}[theorem]{Definition}
\newtheorem{question}[theorem]{Question}
\newtheorem{problem}[theorem]{Problem}
\newtheorem{remark}[theorem]{Remark}
\newtheorem{conjecture}[theorem]{Conjecture}

\def\xxx{\vskip5pt\hrule\vskip5pt}

\def\Cmt#1{\underline{{\sl Comments:}} {\it{#1}}}

\newcommand{\Modp}[1]{
\begin{color}{blue}
 #1\end{color}}
 
 \def\bl#1{\begin{color}{blue}#1\end{color}} 
 \def\red#1{\begin{color}{red}#1\end{color}} 

%\newcommand{\eqname}[1]{\tag{#1}}% Tag equation with name

%%%%%%%%%%%%%%%%%%%%%%%%%
% Alphabet calligraphic %
%%%%%%%%%%%%%%%%%%%%%%%%%
\def\cA{{\mathcal A}}
\def\cB{{\mathcal B}}
\def\cC{{\mathcal C}}
\def\cD{{\mathcal D}}
\def\cE{{\mathcal E}}
\def\cF{{\mathcal F}}
\def\cG{{\mathcal G}}
\def\cH{{\mathcal H}}
\def\cI{{\mathcal I}}
\def\cJ{{\mathcal J}}
\def\cK{{\mathcal K}}
\def\cL{{\mathcal L}}
\def\cM{{\mathcal M}}
\def\cN{{\mathcal N}}
\def\cO{{\mathcal O}}
\def\cP{{\mathcal P}}
\def\cQ{{\mathcal Q}}
\def\cR{{\mathcal R}}
\def\cS{{\mathcal S}}
\def\cT{{\mathcal T}}
\def\cU{{\mathcal U}}
\def\cV{{\mathcal V}}
\def\cW{{\mathcal W}}
\def\cX{{\mathcal X}}
\def\cY{{\mathcal Y}}
\def\cZ{{\mathcal Z}}

\def\C{\mathbb{C}}
\def\F{\mathbb{F}}
\def\K{\mathbb{K}}
\def\L{\mathbb{L}}
\def\G{\mathbb{G}}
\def\Z{\mathbb{Z}}
\def\R{\mathbb{R}}
\def\Q{\mathbb{Q}}
\def\N{\mathbb{N}}
\def\M{\textsf{M}}
\def\U{\mathbb{U}}
\def\P{\mathbb{P}}
\def\A{\mathbb{A}}
\def\fp{\mathfrak{p}}
\def\n{\mathfrak{n}}
\def\X{\mathcal{X}}
\def\x{\textrm{\bf x}}
\def\w{\textrm{\bf w}}
\def\a{\textrm{\bf a}}
\def\k{\textrm{\bf k}}
\def\ee{\textrm{\bf e}}
\def\ovQ{\overline{\Q}}
\def \Kab{\K^{\mathrm{ab}}}
\def \Qab{\Q^{\mathrm{ab}}}
\def \Qtr{\Q^{\mathrm{tr}}}
\def \Kc{\K^{\mathrm{c}}}
\def \Qc{\Q^{\mathrm{c}}}
\newcommand \rank{\operatorname{rk}}
\def\ZK{\Z_\K}
\def\ZKS{\Z_{\K,\cS}}
\def\ZKSf{\Z_{\K,\cS_f}}
\def\ZKSfG{\Z_{\K,\cS_{f,\Gamma}}}

\def\bF{\mathbf {F}}

\def\({\left(}
\def\){\right)}
\def\[{\left[}
\def\]{\right]}
\def\<{\langle}
\def\>{\rangle}

\def\gen#1{{\left\langle#1\right\rangle}}
\def\genp#1{{\left\langle#1\right\rangle}_p}
\def\genPs{{\left\langle P_1, \ldots, P_s\right\rangle}}
\def\genPsp{{\left\langle P_1, \ldots, P_s\right\rangle}_p}

\def\e{e}

\def\eq{\e_q}
\def\fh{{\mathfrak h}}

\def\lcm{{\mathrm{lcm}}\,}

\def\({\left(}
\def\){\right)}
\def\fl#1{\left\lfloor#1\right\rfloor}
\def\rf#1{\left\lceil#1\right\rceil}
\def\mand{\qquad\mbox{and}\qquad}

\def\jt{\tilde\jmath}
\def\ellmax{\ell_{\rm max}}
\def\llog{\log\log}

\def\m{{\rm m}}
\def\ch{\hat{h}}
\def\GL{{\rm GL}}
\def\Orb{\mathrm{Orb}}
\def\Per{\mathrm{Per}}
\def\Preper{\mathrm{Preper}}
\def \S{\mathcal{S}}
\def\vec#1{\mathbf{#1}}
\def\ov#1{{\overline{#1}}}
\def\Gal{{\mathrm Gal}}
\def\Sp{{\mathrm S}}
\def\tors{\mathrm{tors}}
\def\PGL{\mathrm{PGL}}
\def\wH{{\rm H}}
\def\Gm{\G_{\rm m}}

\def\house#1{{%
    \setbox0=\hbox{$#1$}
    \vrule height \dimexpr\ht0+1.4pt width .5pt depth \dp0\relax
    \vrule height \dimexpr\ht0+1.4pt width \dimexpr\wd0+2pt depth \dimexpr-\ht0-1pt\relax
    \llap{$#1$\kern1pt}
    \vrule height \dimexpr\ht0+1.4pt width .5pt depth \dp0\relax}}

\newcommand{\bfalpha}{{\boldsymbol{\alpha}}}
\newcommand{\bfomega}{{\boldsymbol{\omega}}}

\newcommand{\Ch}{{\operatorname{Ch}}}
\newcommand{\Elim}{{\operatorname{Elim}}}
\newcommand{\proj}{{\operatorname{proj}}}
\newcommand{\h}{{\operatorname{\mathrm{h}}}}
\newcommand{\ord}{\operatorname{ord}}

\newcommand{\hh}{\mathrm{h}}
\newcommand{\aff}{\mathrm{aff}}
\newcommand{\Spec}{{\operatorname{Spec}}}
\newcommand{\Res}{{\operatorname{Res}}}

\def\fA{{\mathfrak A}}
\def\fB{{\mathfrak B}}

\numberwithin{equation}{section}
\numberwithin{theorem}{section}

\title[bounds for integral points in orbits over function fields]
{ On effective $\epsilon$-integrality in orbits  of rational maps over function fields and multiplicative dependence}

\author[Jorge Mello]{Jorge Mello}

\address{University of New South Wales. mailing adress:\newline School of Mathematics and Statistics
UNSW Sydney 
NSW, 2052
Australia.} 

\email{j.mello@unsw.edu.au}


\begin{abstract} We give effective bounds for the set quasi-integral points in orbits of non-isotrivial rational maps over function fields under some conditions, generalizing previous work of Hsia and Silverman (2011)
for orbits over function fields of characteristic zero. We then use this to prove finiteness results for algebraic functions whose orbit under a rational function has multiplicative dependent elements modulo rings of $S$-integers, generalizing recent results over number fields.
\end{abstract}

\maketitle
\section{Introduction} Let $K$ be a function field of a smooth projective curve over an algebraically closed field of characteristic $0$, endowed as usual with a set $M_K$ of absolute values(places) satisfying the product formula, $S$ a finite subset of places of $M_K$, and $\epsilon >0$.
An element $x \in K$ is said to be \textit{quasi-$(S,\epsilon)$-integral} if 
\begin{center}
 $\displaystyle\sum_{v \in S}  \log (\max \{|x|_v, 1 \}) \geq \epsilon h([x,1])$,
 \end{center} where $h$ is the  absolute logarithmic height in $\mathbb{P}^1(K)$ and $[x,1] \in \mathbb{P}^1(K)$.
 
 Let $ \phi \in K(z)$ a of rational functions of degree at least $2$, let $P \in K$ and let 
\begin{center}
$\mathcal{O}_{\phi}(P)= \{ \phi^{(n)} (P) |  n \in \mathbb{N} \}$
\end{center}
denote the forward orbit of $P$ under $\phi$.
When $K$ is a number field and $\phi^2= \phi \circ \phi \notin K[z]$, Hsia and Silverman proved \cite{HS} that the number of quasi-$(S, \epsilon)$- integral points in the orbit of a point $P$ with infinite orbit is bounded by a constant depending only on $\phi, \hat{h}_{\phi}(P), \epsilon, S$, and $[K:\mathbb{Q}]$ ( see Section 2 for the correspondent definitions). We also note that these results, according to \cite[Remark 1]{HS}, have some applications as  the existence of quantitative estimates for the size of Zsigmondy sets for such orbits and their primitive divisors, as well for quantitative versions of a dynamical local-global principle in orbits on the projective line. This research was also used to prove finiteness of multiplicatively dependent iterated values by rational functions in \cite{BOSS}. 

In this present paper we generalize this bound for cases over function fields. This is presented all over Section 2. 
Making use of such results, this work is also place for the study of multiplicative dependence modulo $S$-integers for elements in orbits of rational functions over function fields, generalizing the results of \cite{BOSS} for this context, but now also with effective bounds. This is done in Section 3.  Sometimes it is only possible to bound, although effectively, the height of algebraic functions studied, instead of their cardinality. This is related with the fact that even the so called \textit{Northcott finiteness property} (which is a particular case of finiteness for elements with multiplicative dependent orbit elements) fails over function fields. One can see once again that the ambient of function fields is relevantly different to number field one. For this, we  made use of specific tools from the function field case. Namely, a recent version of effective Roth type of theorem over function fields due to Wang \cite{W1}, a certain finiteness property for canonical heights due to Baker \cite{B}, and some effective results for superellitic equations over function fields in one variable.

 \section{Effective bounds for quasiintegral points in orbits over function fields}
\subsection{Canonical Heights, Distance and dynamics on the projective line} We always assume that $K$ is a fixed function field of a curve over an algebraically closed field of characteristic $0$ and $K(z)$ is the field of rational functions over $K$ for the rest of the paper. We identify $K \cup \{ \infty \} = \mathbb{P}^1(K)$ by fixing an affine coordinate $z$ on $\mathbb{P}^1$, so $\alpha \in K$ is equal to $[\alpha, 1] \in \mathbb{P}^1(K)$, and the point at infinity is $[1,0]$. In this way, we assume $z$ is the first left coordinate for points in $\mathbb{P}^1$,
and with respect to this affine coordinate, we identify rational self-maps of $\mathbb{P}^1$ with rational functions in $K(z)$. 
 
 If $P= [x_0,...,x_N] \in \mathbb{P}^N(K)$, the naive logarithmic height is given by 
\begin{center}$h(P)= \sum_{v \in M_K}  \log(\max_i |x_i|_v)$, \end{center}
where, $M_K$ is the set of places of $K$, and for each $v \in M_K$, $|.|_v$ denotes the corresponding absolute values on $K$ satisfying the product formula, and can be extended to any algebraic closure and respective completions of $K$, so that $h$ can be well defined on $\overline{K}$. Also, we write $K_v$ for the completion of $K$ with respect to $|.|_v$, and we let $\hat{K}_v$ denote the completion of an algebraic closure of $K_v$.
Initially, we also recall  that one can define the convergent limit $\hat{h}_f(\alpha)=\lim_{n \rightarrow \infty} \dfrac{h(f(\alpha))}{d^n}$ for any $\alpha \in K$, called the \textit{canonical height} associated with $f$, which also satisfies $\hat{h}_f(f(\alpha))=d\hat{h}_f(\alpha)$, and that $\alpha$ has infinite orbit (is not \textit{preperiodic}) if and only if $\hat{h}_f(f(\alpha))>0$( see \cite[Theorem 3.20]{Silv2}).

  For each $v \in M_K$, we let $\rho_v$ denote the chordal metric defined on $\mathbb{P}^1(\hat{K}_v)$, where we recall that for 
  $[x_1,y_1], [x_2,y_2] \in \mathbb{P}^1(\hat{K}_v)$,
  \begin{center} $  
\rho_v([x_1,y_1], [x_2,y_2]) = 
     \dfrac{|x_1y_2 - x_2y_1|_v}{\max \{ |x_1|_v,|y_1|_v \} \max \{ |x_2|_v,|y_2|_v \}}.
      $ \end{center}
 \begin{definition} The \textit{logarithmic chordal metric function} 
 \begin{center}
  $\lambda_v : \mathbb{P}^1(\hat{K}_v) \times \mathbb{P}^1(\hat{K}_v) \rightarrow \mathbb{R} \cup \{ \infty \}  $
  \end{center} is defined by
  \begin{center}
   $\lambda_v([x_1,y_1], [x_2,y_2]) =- \log \rho_v([x_1,y_1], [x_2,y_2]).$
  \end{center}
  \end{definition}
  
  It is a matter of fact that $\lambda_v$ is a particular choice of an \textit{arithmetic distance function} as defined by Hsia and Silverman \cite{HS} over number fields, which is
  a local height function $\lambda_{\mathbb{P}^1 \times \mathbb{P}^1, \Delta}$, where $\Delta$ is the diagonal of $\mathbb{P}^1 \times \mathbb{P}^1$.
  The logarithmic chordal metric and the usual metric can relate in the following way.
 \begin{lemma}\label{lem2.2}
  Let $v \in M_K$ and let $\lambda_v$ be the logarithmic chordal metric on $\mathbb{P}^1(\mathbb{C}_v)$. 
  Then for $x,y \in \mathbb{C}_v$ the inequality $\lambda_v(x,y) > \lambda_v(y,\infty) $ implies
  \begin{center}
   $ \lambda_v(y,\infty) \leq \lambda_v(x,y) + \log |x-y|_v \leq 2 \lambda_v(y,\infty)$.
  \end{center}

 \end{lemma}\begin{proof}
 The proof works in the same way as the proof over number fields appearing in \cite[Lemma 3]{HS}.
 \end{proof}
 
  Now, let $\phi: \mathbb{P}^1 \rightarrow \mathbb{P}^1$ be a rational map of degree $d \geq 2$ defined over $K$.
 In this situation we let \begin{center}$\phi^{(n)}=\phi \circ ... \circ \phi$ \end{center} with $\phi^{(0)}=$Id. For simplicity $\phi^{(n)}=\phi^n$.

For a point $P \in \mathbb{P}^1$, the $\phi$-orbit of $P$ is defined as 
\begin{center}
 $\mathcal{O}_{\phi}(P)= \{ \phi^{(n)}(P) | n \geq 0 \}.$
 \end{center}
 The point $P$ is called \textit{preperiodic for $\phi$} if $\mathcal{O}_{\phi}(P)$ is finite. 
 
 We set $\text{Wander}_K(\phi)=\{P \in \mathbb{P}^1(K): P \text{ is not preperiodic for } \phi\}$.
 
 We recall that for $P=[x_0, x_1] \in \mathbb{P}^1(K)$ the height of $P$ is 
 \begin{center}
  $h(P)= \sum_{v \in M_K}  \log(\max \{ |x_1|_v, |x_1|_v \}$.
 \end{center} And using the definition of $\lambda_v$, we see that
 \begin{center}
  $h(P)= \sum_{v \in M_K} \lambda_v(P, \infty) + O(1)$.
 \end{center}
 
 For a polynomial $f= \sum a_i z^i$ and an absolute value $v \in M_K$, we define 
  $|f|_v=\max_i \{ |a_i|_v \}$  and \begin{center} $h(f)= \sum_{v \in M_K}  \log |f|_v$.
 \end{center}
 Given a rational function $\phi(z)= f(z)/g(z) \in K(z)$ of degree $d$ written in normalized form, let us write $f(z)=\sum_{i \leq d} a_i z^i, g(z)=\sum_{i \leq d} b_i z^i $ with $a_d$ and $b_d$ different from zero,
 and $f$ and $g$ relatively prime in $K[z]$. 
 
 For $v \in M_K$, we set $|\phi|_v= \max \{ |f|_v, |g|_v \}$, and then the height of $\phi$ is defined by
 \begin{center}
  $h(\phi):= \sum_{v \in M_K}  \log |\phi|_v$.
 \end{center}

 \begin{prop}\label{prop2.3}
  Let $\phi$ be a  rational function with $\deg \phi= d \geq 2$. Then for all $n \geq 1$, we have
  \begin{center}
   $h(\phi^n) \leq \left(\dfrac{d^n-1}{d-1}\right)h(\mathcal{F}) + d^2\left(\dfrac{d^{n-1}-1}{d-1}\right)\log 8$.
  \end{center}
 \end{prop}
 \begin{proof}
\cite[Proposition 5 (d)]{HS}.
  \end{proof}
  
  \begin{lemma}\label{lem2.4} For a rational map $\phi: \mathbb{P}^1 \rightarrow \mathbb{P}^1$ of degree $d \geq 2$ defined over $K$ and $L= \mathcal{O}_{\mathbb{P}^1}(1)$, it is true that
  
  (a) $|h(\phi (P))- dh(P)| \leq c_1h(\phi) + c_2$.
  
  (b) $\hat{h}_{\phi}(P)= \lim_n h(\phi^{(n)}(P))/d^n$.
  
  (c) $|\hat{h}_{\phi}(P) - h(P)| \leq c_3h(\phi) + c_4$.
  
  Where $c_1, c_2, c_3$ and $c_4$ above depend only on $d$.
  \end{lemma}\begin{proof}
  This is stated in \cite[Proposition 6]{HS} over number fields. For (a), the same procedure given in \cite[Proposition A]{BMZ} works over the referred function fields. For (b) and (c), the proof of the existence of the limit defining the canonical height (\cite[Theorem 3.20]{Silv2}) together with (a) yields the desired.
  \end{proof}
  
  \begin{lemma}\cite[Theorem 1.6]{B}\label{lem2.5}
Let $\varphi(z) \in K(z)$ of degree at least $2$, and assume that $\varphi$ is not isotrivial, and $\hat{h}_\varphi$ is the canonical height associated with $\varphi$. Then there exists $\varepsilon>0$ ( depending on $K$ and $\varphi$) such that the set 
$$
\{ P \in \mathbb{P}^1(K): \hat{h}_\varphi(P) \leq \varepsilon\}
$$ is finite.
\end{lemma}

  \subsection{A distance estimate and an effective version of Roth's Theorem}
We will state three results that will be needed to prove our main theorems. The first one is a result  that gives explicit estimates for the dependence on local heights of points and function.
 
 Let us recall that, for a rational function $f(z)$, $P \neq \infty$ and $f(P) \neq \infty$, the \textit{ramification index} of $f$ at $P$ is defined as the order of $P$ as a zero of the rational function $f(z) - f(P)$, i.e., 
 \begin{center}
  $e_P(f)=$ ord$_P(f(z) - f(P))$.
 \end{center} If $P= \infty$, or $f(P)=\infty$, we change coordinates through a linear fractional transformation $L$, such that $L^{-1}(P)=\beta \neq \infty, L^{-1}(f(L(\beta))) \neq \infty$, and define $e_P(f)=e_{\beta}(L^{-1} \circ f \circ L)$. It will not depend on the choice of $L$. We say that $f$ is \textit{totally ramified} at $P$ if $e_P(f)=\deg f$. It is also an exercise to show that 
 \begin{center}
 $e_P(g \circ f)=e_P(f) e_g(f(P))$
 \end{center} for every $f,g$ rational functions and $P \in K \cup \{ \infty \}$.

 The result is as follows.
 
 \begin{lemma}\label{lem2.6}
  Let $\psi \in K(z)$ be a nontrivial rational function, let $S \subset M_K$ be a finite set of absolute values on $K$, each extended in some way to $\bar{K}$, and let $A, P \in \mathbb{P}^1(K)$. Then 
 \begin{align*}
  \sum_{v \in S} \max\limits_{A^{\prime} \in \psi^{-1}(A)} e_{A^{\prime}}(\psi) & \lambda_v (P, A^{\prime}) \geq \\&
  \sum_{v \in S}  \lambda_v (\psi(P),A) - O(h(A) + h(\psi) +1),
 \end{align*} where the implied constant depends only on the degree of the map $\psi$.
\end{lemma}
\begin{proof} This is stated in \cite[Proposition 7]{HS} over number fields. Its proof uses a higher dimensional version of Lemma \ref{lem2.4}(b) applied for maps in dimension 1, thus Lemma \ref{lem2.4} is enough and works over function fields. More importantly, the proof uses strong distribution value theorems related with inverse function theorem due to Silverman. It was found an error in these proofs, which were accordingly corrected in \cite[Sections 4 and 5]{MS}, that work over global fields.
\end{proof}

  \begin{lemma} \cite[Lemma 9]{HS}\label{lem2.7} Fix an integer $d \geq 2$. Then there exist two positive constants $\kappa_1 >0$ and $ 0 < \kappa_2<1$ depending only on $d$ such that for all rational functions $\phi:\mathbb{P}^1(K) \rightarrow \mathbb{P}^1(K)$, all points $Q$ that are not exceptional for $\phi$, all integers $m \geq 1$, and all $P \in \phi^{-m}(Q)$, we have 
   that \begin{center}
             $e_P(\phi^m) \leq \kappa_1 (\kappa_2d)^m$ for any $m \geq 0$.
            \end{center}
\end{lemma}
\begin{proof}
 The proof of (\cite[Lemma 3.52]{Silv2}) works here, since the Riemann-Hurwitz formula works for this context as well.

\end{proof}

The third result is the following effective version of Roth's theorem  over function fields due to Wang.

\begin{lemma} \cite{W1}\label{lem2.8}
 Let $S$ be a finite subset of $M_K$. We assume that each place in $S$ is extended to $\bar{K}$ in some fashion. Assume that for each $v \in S$, we have an element $\beta_v \in \bar{K}$. Then, for any $\mu >2$, the elements $x \in K$ satisfying 
$$ \sum_{v \in S}  \log^+ |x - \beta_v|_v^{-1} \geq \mu h(x)$$ have their heights bounded by an effective constant depending on $\mu,|S|,$ the genus of $K$, and the elements $\beta_v$.
\end{lemma}

\subsection{A bound for the number of quasiintegral points in an orbit}
In this section, we show explicit bounds for the number of $S$-integral points in a given orbit of a wandering point for a dynamical system of rational functions
extending previous work by Hsia and Silverman \cite{HS}.

The next quantitative theorem generalizes Theorem 11 of Hsia and Silverman \cite{HS} to function fields of zero characteristic. The definitions and strategy of the proof are inspired by their ideas with diophantine approximation.
\begin{theorem}\label{th2.9}
 Let $\phi \in K(z)$ be a non-isotrivial rational function of respective degree $d\geq 2$, and $P \in \mathbb{P}^1(K)$ not preperiodic for $\phi$. Fix $A \in \mathbb{P}^1(K)$ which is not an exceptional point of $\phi$. For any finite set of places $S \subset M_K$ and any constant $1 \geq \epsilon >0$, define a set of nonnegative integers by
\begin{center}
 $\Gamma_{\phi,S}(A, P, \epsilon) := \{ \phi^{(n)}(P) : \sum_{v \in S}  \lambda_v (\phi^n(P), A) \geq \epsilon \hat{h}_{\phi}(\phi^n (P)) \}$.
 \end{center}
 (a) There exist effective constants
 \begin{center}
  $\gamma_1= \gamma_1(\phi, \epsilon,|S|, K,A)$ and $\gamma_2= \gamma_2(\phi, \epsilon,|S|, K,A)$
 \end{center} such that
 \begin{center}
  $ \left\{\phi^{(n)}(P) \in \Gamma_{\phi,S}(A, P, \epsilon) : n >  \gamma_1 + \log_{d}^+ \left(\dfrac{\hat{h}_{\phi}(A)+h(\phi)}{\hat{h}_{\phi}(P)}\right)\right\} $ 
 \end{center}has bounded height from above by $\gamma_2$
 \newline (b) If $P$ is not $\phi$-preperiodic,  there is an effective constant $\gamma_3(\phi, \epsilon,|S|, K,A)$ that is independent of $P$ such that
\begin{center}
 $\displaystyle\max_{P} \{ n \geq 0 : \phi^{(n)}(P) \in  \Gamma_{\phi,S}(A, P, \epsilon)\} \leq \gamma_3+ \log_{d}^+ \left(\dfrac{h(\phi)}{\inf_{\hat{h}_\phi(P)>0}\hat{h}_{\phi}(P)}\right)$.
\end{center}

\end{theorem}

\begin{proof}
 For simplicity, we write $\Gamma_{S}(\epsilon)$ instead of $\Gamma_{\Phi,S}(A, P, \epsilon)$.
Taking $\kappa_1$ and $\kappa_2 < 1$ the constants from Lemma \ref{lem2.7}, we choose $m \geq 1$ minimal such that
$\kappa^m_2 \leq \epsilon/5\kappa_1$. Then $\kappa_1, \kappa_2$ and $ m$ depend only on $d$ and on $\epsilon$. 

If $n \leq m$ for all $n$ such that $\phi^{n}(P) \in \Gamma_{S}(\epsilon)$, then
\begin{center}
 $\# \Gamma_{S}(\epsilon) \leq m \leq \dfrac{\log (5\kappa_1) + \log (\epsilon^{-1})}{\log (\kappa_2^{-1})} + 1$,
\end{center} which is in the desired form.
If there is an $n$ with $\phi^{n}(P) \in \Gamma_{S}(\epsilon)$ such that $n > m$, we fix $n$ for instance. Then by definition of $\Gamma_{S}(\epsilon)$ we have
\begin{equation}\label{eq2.1}
 \epsilon \hat{h}_{\phi}(\phi^n (P)) \leq  \sum_{v \in S}  \lambda_v (\phi^n(P), A).
\end{equation}
We can write $\phi^n= \phi^m \circ \phi^{n-m}$ and $\psi=\phi^m$.
\newline \newline For our chosen $m$, we denote
\begin{center}
 $\textbf{e}_m:= \max\limits_{A^{\prime} \in \psi^{-1}(A)} e_{A^{\prime}}(\psi)$.
\end{center} By Lemma \ref{lem2.7} and our choice of $m$, we notice that
\begin{center}
 $\textbf{e}_m \leq \kappa_1 (\kappa_2)^m \deg \psi \leq \epsilon \deg \psi/5$
\end{center}

Therefore, Lemma \ref{lem2.6} yields, for $Q \in \mathbb{P}^1(K)$, that
\begin{equation}\label{eq2.2}
 \sum_{v \in S}  \lambda_v ( \psi(Q),A) - O(h(A) + h(\psi) + 1)   \leq \textbf{e}_m \sum_{v\in S} \max\limits_{A^{\prime} \in \psi^{-1}(A)}  \lambda_v(Q, A^{\prime}).
\end{equation}
Gathering \ref{eq2.1} and \ref{eq2.2} with $Q:= \phi^{n-m}(P)$, we obtain that
\begin{center}
 $\epsilon \hat{h}_{\phi}(\phi^n (P)) \leq \textbf{e}_m \sum_{v\in S} \max\limits_{A^{\prime} \in \psi^{-1}(A)}  \lambda_v(\phi^{n-m}(P), A^{\prime}) + O(h(A) + h(\psi) + 1)$,
\end{center} where the involved constants depend only on the degree  $d^m$, $d$ and on $\epsilon$.

For each $v \in S$, we choose $A_v^{\prime} \in \psi^{-1}(A)$ such that
\begin{center}
 $ \lambda_v(\phi^{n-m}(P), A_v^{\prime})=\max\limits_{A^{\prime} \in \psi^{-1}(A)}  \lambda_v(\phi^{n-m}(P), A^{\prime})$,
\end{center} so that
\begin{center}
 $\epsilon \hat{h}_{\phi}(\phi^n (P)) \leq \textbf{e}_m \sum_{v\in S}  \lambda_v(\phi^{n-m}(P), A_v^{\prime}) + O(h(A) + h(\psi) + 1)$.
\end{center}
For instance, we can assume that $z(A^{\prime}) \neq \infty$ for all $A^{\prime} \in \psi^{-1}(A)$. If this is not the case, we use $z$ for some of the $A^{\prime}$ and $z^{-1}$ for the others.

Let $S^{\prime} \subset S$ be the set of places in $S$ defined by
\begin{center} 
 $S^{\prime}= \{ v \in S ; \lambda_v(\phi^{n-m}(P), A_v^{\prime}) > \lambda_v (A_v^{\prime}, \infty) \}$. 
\end{center}  

Set $S^{\prime \prime}:= S - S^{\prime}$. Applying Lemma \ref{lem2.2} to the places in $S^{\prime}$ and using the definition of $S^{\prime \prime}$ we find that
\begin{align*}
 \epsilon \hat{h}_{\phi}(\phi^n (P))& \leq \textbf{e}_m \displaystyle\sum\limits_{v\in S}
  \lambda_v(\phi^{n-m}(P), A_v^{\prime}) + O(h(A) + h(\psi) + 1)\\
&\leq \textbf{e}_m \displaystyle\sum\limits_{v\in S^{\prime}} (2 \lambda_v(A_v^{\prime},\infty)-\log|z(\phi^{n-m}(P))- z(A_v^{\prime})|+ \log l_v)\\
& \qquad \qquad 
 +\textbf{e}_m \displaystyle\sum\limits_{v\in S^{\prime \prime}} (\lambda_v(A_v^{\prime},\infty) + \log l_v) + O(h(A)+ h (\psi) + 1)\\
 &\leq \textbf{e}_m \displaystyle\sum\limits_{v\in S^{\prime}}  \log|z(\phi^{n-m}(P))- z(A_v^{\prime})|^{-1}\\
 & \qquad \qquad + \textbf{e}_m \displaystyle\sum\limits_{v\in S} (2\lambda_v(A_v^{\prime},\infty) + \log l_v) + O(h(A)+ h (\psi) + 1).
\end{align*}
Now using Lemma \ref{lem2.4} it can be checked that
\begin{align*}
\displaystyle\sum\limits_{v\in S}  \lambda_v(A_v^{\prime},\infty) & \leq \displaystyle\sum\limits_{A^{\prime} \in \psi^{-1}(A)} \displaystyle\sum\limits_{v\in S}  \lambda_v(A^{\prime},\infty)\\
& \leq \displaystyle\sum\limits_{A^{\prime} \in \psi^{-1}(A)} h(A^{\prime})\\ &\leq \displaystyle\sum\limits_{A^{\prime} \in \psi^{-1}(A)} \hat{h}_{\phi}(A^{\prime}) + O(h(\phi)+1)\\ 
 &=\displaystyle\sum\limits_{A^{\prime} \in \psi^{-1}(A)} (\deg \psi)^{-1} \hat{h}_{\phi}(\psi(A^{\prime})) + O(h(\phi)+1)\\
&\leq \displaystyle\sum\limits_{A^{\prime} \in \psi^{-1}(A)} (\deg \psi)^{-1} \hat{h}_{\phi}(A) + O(h(\phi)+1)\\
&\leq \hat{h}_{\phi}(A) + O(h(\phi) + 1). 
\end{align*}
The constants depend only on $m$ and $d$.
Also, from Proposition \ref{prop2.3} it follows that $h(\psi) = O(h(\phi) +1)$.

All the inequalities above together imply that
\begin{center}
 $\epsilon (\hat{h}_{\phi}(\phi^n (P)) \leq \textbf{e}_m (\displaystyle\sum\limits_{v\in S^{\prime}}  \log|z(\phi^{n-m}(P))- z(A_v^{\prime})|^{-1})+ O(\hat{h}_{\phi}(A) +h(\phi) + 1)$.
\end{center}
Let us set some definitions in order to apply Wang's Theorem. We define $\beta_v:= A_v^{\prime}$ and analyze the points
$x = \phi^{n-m}(P)$ for $\phi^n(P) \in \Gamma_S(\epsilon)$. Applying Lemma \ref{lem2.8} for the set of places $S^{\prime}$ and $\mu=5/2$, yields that there exist a constant $r_1$ depending only on $K, \phi, |S|$ and $\epsilon$ such that the set of $\phi^n(P) \in \Gamma_S(\epsilon)$ with $n>m$ can be written as a union
\begin{center}
 $\{ \phi^n(P) \in \Gamma_S(\epsilon) : n>m\}= T_1 \cup T_2 \cup T_3$
\end{center} such that \newline \newline
$  T_1$ has all its elements with height bounded from above by $r_1$, \newline \newline
$T_2 = \{\phi^n(P) \in \Gamma_S(\epsilon): n >m,  \displaystyle\sum\limits_{v\in S^{\prime}} d_v \log|z(\phi^{n-m}(P))- z(A_v^{\prime})|^{-1} \leq \frac{5}{2} h(\phi^{n-m}(P)) \},$ \newline \newline

We already have a bound for the height of $T_1$.
 We consider the set $T_2$. Again using Lemmas \ref{lem2.4} we derive
\begin{center}
 $h(\phi^{n-m}(P)) \leq \hat{h}_{\phi}(\phi^{n-m}(P)) + c_1 h(\phi) + c_2\newline \newline =d^{n-m}\hat{h}_{\phi}(P)+ c_1 h(\phi) + c_2$,
\end{center}
and then, for $n$ with $\phi^n(P) \in T_2$, using that $\textbf{e}_m \leq \epsilon \deg \psi /5$
\begin{align*}
 \epsilon \hat{h}_{\phi}(\phi^n(P))&=\epsilon d^n \hat{h}_{\Phi}(P)\\
 &\leq \textbf{e}_m (\displaystyle\sum\limits_{v\in S^{\prime}}  \log|z(\phi^{n-m}(P))- z(A_v^{\prime})|^{-1})+ c_{13}(\hat{h}_{\phi}(A) +h(\phi) + 1)\\
 &\leq (\epsilon \frac{\deg \psi}{5})\frac{5}{2}(d^{n-m})\hat{h}_{\phi}(P)+ c_{10}(\hat{h}_{\phi}(A)+ h(\phi) + 1)\\
  &= \frac{\epsilon}{2}d^n\hat{h}_{\phi}(P)+ c_{14}(\hat{h}_{\phi}(A) +h(\phi) + 1).
  \end{align*} Thus \newline \newline
  $\frac{\epsilon}{2}d^n\hat{h}_{\phi}(P) \leq c_{14}(\hat{h}_{\phi}(A) +h(\phi) + 1)$, which implies that
  \begin{center}
   $\frac{\epsilon}{2}d_1^n \hat{h}_{\phi}(P) \leq c_{14}(\hat{h}_{\phi}(A) +h(\phi) + 1)$,
  \end{center} equivalent to
  \begin{center}
   $n \leq c_{15}+ \log_{d}^+ \left(\dfrac{\hat{h}_{\phi}(A)+h(\phi)}{\hat{h}_{\phi}(P)}\right)$.
  \end{center}
  We observe that the set $r_1$ does not depend on the point, so the elements in $T_1$ have height  bounded  independently of $P$ by the constant $r_1$.
  We also note that the quantity
  \begin{center}
   $\hat{h}^{\text{min}}_{\phi,K} := \inf \{\hat{h}_{\phi}(P) : P \in \mathbb{P}^1(K)$ is not preperiodic for $\Phi  \}$
  \end{center} is strictly positive. This is a consequence of Lemma \ref{lem2.5}. For $\phi^n(P) \in T_1$, we can see from this and Lemma \ref{lem2.4} that $d^n\hat{h}_\phi(P)=\hat{h}_\phi(\phi^n(P))\leq r_1 +O(h(\phi)+1)$, and thus 
  $$
  n \leq \log_d \left(\dfrac{r_1+O(h(\phi)+1)}{\hat{h}^{\text{min}}_{\phi,K}} \right)
  $$ in this case.

Therefore, $\max\{ n : \phi^n(P)  \in (T_1 \cup T_2) \} $ can be bounded independently of $P$.
\end{proof}

\begin{corollary}\label{cor2.10}
 Let $S \subset M_K$ be a finite set of places, let $R_S$ be the ring os $S$-integers of $K$, and let $2\leq d$. Then, there is an effective constant
 $\gamma= \gamma(\phi,|S|, K)$ such that for all $ \phi \in K(z)$ non-isotrivial rational maps of  degrees $d\geq 2$ with $\phi^2 \notin \bar{K}[z]$, all $P \in \mathbb{P}^1(K)$ that are not preperiodic for $\phi$,
 the number of $S$-integers in the $\phi-$orbit of $P$ is bounded by
 \begin{center}
  $\# \{n \geq 1 ; z(\phi^n(P)) \in R_S \} \leq \gamma + \log^+_{d}\left(\dfrac{h(\phi)}{\hat{h}_{\phi}(P)}\right)$.
 \end{center}

\end{corollary}
\begin{proof}
An element $\alpha \in K$ is in $R_S$ if and only if $|\alpha|_v \leq 1$ for all $v \not\in S$, or equivalently, if and only if
\begin{center}
 $h(\alpha) = \sum_{v \in S} \log \max \{ |\alpha|_v,1 \}$.
\end{center} Another fact is that
\begin{center}
 $\log \max \{ |\alpha|_v,1 \} \leq \lambda_v(\alpha, \infty)$.
\end{center} This implies for $\alpha \in R_S$ that $h(\alpha) \leq \sum_{v \in S} \lambda_v(\alpha, \infty)$.

Let $ n \geq 1$ satisfy $z(\phi^n(P)) \in R_S$. Then
\begin{center}
 $h(\phi^n(P)) \leq \sum_{v \in S}  \lambda_v(\phi^n(P), \infty)$.
\end{center}
Lemmas \ref{lem2.4} tell us that
\begin{center}
 $h(\phi^n(P)) \geq \hat{h}_{\phi}(\phi^n(P)) - c_3h(\phi) - c_4=\deg(\phi^n)\hat{h}_{\phi}(P) - c_3h(\phi) - c_4$,
\end{center} which implies that
\begin{center}
 $d^n \hat{h}_{\phi}(P) - c_3h(\phi) - c_4 \leq \sum_{v \in S} d_v \lambda_v(\phi^n(P), \infty)$.
\end{center}
The rest of  the proof is divided in two cases:
First one, when 
\begin{center}
 $d^n\hat{h}_{\phi}(P) \leq 2c_3h(\phi) + 2 c_4$.                
\end{center}In this case, 
\begin{center}
 $n \leq \log_{d}^+ \left(  \dfrac{2c_3h(\phi) + 2c_4}{\hat{h}_{\phi}(P)}  \right)$.
\end{center} In the second case , $d^n\hat{h}_{\phi}(P) \geq 2c_3h(\phi) + 2 c_4$. Therefore
\begin{center}
 $\sum_{v \in S}  \lambda_v(\phi^n(P), \infty) \geq \frac{1}{2}d^n \hat{h}_{\phi}(P)=\frac{1}{2}\hat{h}_{\phi}(\phi^n(P))$.
\end{center} Now  Theorem \ref{th2.9} with $\epsilon =1/2, A= \infty$ ($\infty$ is not exceptional for $\phi$) tells us that $n$ is at most
\begin{center}
 $ \gamma + \log^+_{d} \left( \dfrac{h(\phi) + \hat{h}_{\phi}(\infty)}{\hat{h}_{\phi}(P)} \right)$,
\end{center} for an effective constant $\gamma$ depending only on $K, \phi$ and $ |S|$. Both bounds are on the desired form since $\hat{h}_{\phi}(\infty)\leq  h(\infty) +O(1)=0 + O(1)$.
\end{proof}

 \begin{remark}
 Theorem \ref{th2.9} delivers, in particular, under its conditions, an explicit upper bound for
 \begin{center}
 $\# \{ n \geq 1 ; \dfrac{1}{\phi^n(P) - A} $ is quasi-$(S, \epsilon)$-integral $ \}$.
  
 \end{center} 

 \end{remark}
 
 \begin{corollary}
  Under the hypothesis of Theorem \ref{th2.9},
  \begin{center}
   $\displaystyle\lim_{n \rightarrow \infty} \dfrac{\lambda_v(\phi^n(P),A)}{d^n}=\displaystyle\lim_{n \rightarrow \infty} \dfrac{\lambda_v(\phi^n(P),A)}{\hat{h}_\phi(\phi^n(P))}=0$ for every $v \in M_K$.
  \end{center}
\end{corollary}
\begin{proof}
 Applying Theorem \ref{th2.9} for the set of places that contains just the place $v$, we conclude that for every natural $n$ big enough, it will be true that 
 \begin{center}
  $\dfrac{\lambda_v(\phi^n(P),A)}{d^n} \leq \epsilon \hat{h}_{\phi}(P)$.
 \end{center} Choosing $\epsilon$ sufficiently small, the result is proven.

\end{proof} 
\section{Multiplicative dependence in orbits over function fields}
\subsection{ $S$-units, algebraic dynamics, and multiplicative dependence.}
\begin{lemma}\label{lem3.1}
Let $\varphi(z) \in K(z)$  suppose that $|\varphi^{-1}(\infty)| \geq 3$. Then $$\{ f \in K : \varphi(f) \in R_S\} $$ is finite, and each $f$ in this set has height bounded by a constant $C(\varphi,K,|S|)$
\end{lemma}
\begin{proof} This is \cite[Theorem 12(i)]{HSW}.
\end{proof} The following is a version of \cite[Theorem 1.2]{BOSS} for function fields.
\begin{theorem}\label{th3.2}  Let $\varphi(z) \in K(z)$ of degree at least two, and assume that $\varphi$ is not isotrivial. Then \newline  (a) If $|\varphi^{-1}(\{0,\infty \})| \geq 3$, then $$\{ f \in K : \varphi(f) \in R^*_S\} $$ is finite, and its elements have height bounded from above effectively in terms of $\varphi,|S|$ and $K$. In other words, if the referred set is infinite, then $\varphi$ has one of the following forms:
$$
\varphi(X)=f(X-g)^{\pm d}, f \neq 0, \text{ or } \varphi(X)=f(X-g)^{ d}/(X-h)^{ d}, f(g-h) \neq 0.
$$
(b) Let $$\mathcal{F}_2(K,\varphi,R_S^*)=\{(n,\alpha) \in \mathbb{Z}_{\geq 2} \times \text{Wander}_K(\varphi): \varphi^{(n)}(\alpha) \in R_S^* \}.$$ If $(n,\alpha) \in \mathcal{F}_2(K,\varphi,R_S^*)$, then $n$ belongs to a finite set bounded by an explicit constant depending on $\varphi,K,|S|$ and $\inf_{\hat{h}_\varphi(P)>0}\hat{h}_\varphi(P)$, and $\alpha$ to a set of  height bounded effectively from above effectively in terms of $\varphi,|S|$ and $K$, unless  $\varphi$ has the form $\varphi(X)=fX^{\pm d}$.

\end{theorem}
\begin{proof} For (a), such fact follows from the ideas in the proof of \cite[Proposition 1.5(a)]{KLSTYZ}, namely, 
by our hypothesis, the function $\psi(z):= \varphi(z)+1/\varphi(z)$ satisfies the hypothesis of the previous Lemma, thus $\{ f \in K : \psi(f) \in R_S\}$ is finite.
But $\psi(\beta) \in R_S$ whenever $\varphi(\beta) \in R^*_S$, hence it follows that $\{ f \in K : \varphi(f) \in R^*_S\}$ is finite, and each $f$ in it has height bounded by $C(\varphi+1/\varphi,K,|S|)$.

For (b), we consider the well-defined map $$\mathcal{F}_2(K,\varphi,R_S^*) \rightarrow \{ f \in K : \varphi^{(2)}(f) \in R^*_S\} $$ sending $(n,\alpha)$ to $\varphi^{(n-2)}(\alpha)$. By (a), if $\varphi^{(2)}$ is in the same conditions of $\varphi$ in (a), then there are only finitely many possibilities for such $\varphi^{(n-2)}(\alpha)$, whose heights are bounded by $C(\varphi^{(2)}+1/\varphi^{(2)},K,|S|)$. Hence we also obtain an effective bound $C_1(\varphi,K,|S|)$ for the canonical heights of such points, and thus also for the heights of the referred $\alpha$'s. Taking a certain $\varphi^{(n-2)}(\alpha)$ among such possibilities, and using by Lemma \ref{lem2.5}  that $\inf_{\hat{h}_\varphi(P)>0} \hat{h}_\varphi(P)$ exists, we have by the same calculations in the proof of \cite[Lemma 2.3]{BOSS} that $n$ is bounded by 
$2+ \log \left( \frac{C_1(\varphi,K,|S|)}{inf_{\hat{h}_\varphi(P)>0} \hat{h}_\varphi(P)} \right)$, unless $\varphi^{(2)}$ has one of the special forms from (a). Using the Riemann-Hurwitz formula, which is valid in characteristic zero, the same method carried out in the end of the the proof of \cite[Theorem 1.2]{BOSS} can be performed, concluding the desired results.

\end{proof} The next result generalizes \cite[Theorem 1.3]{BOSS} to function fields.

\begin{theorem}\label{th3.3}Let $r,s \in \mathbb{Z}$ with $rs \neq 0$, and set $$ \rho=\dfrac{\log (|s|/|r|)}{\log d}+1$$ Let $\phi \in K(z)$ not isotrivial with degree $d\geq 2$. Assume that $0$ is not a periodic point for $\phi$ and that $|\phi^{-1}(\{0,\infty \})| \geq 3$. We let 
$\mathcal{E}_\rho(K,\phi, S, r,s)$ to be $$ \left\{(n,k,f,u) \in \mathbb{Z}_{\geq \rho} \times \mathbb{Z}_{\geq 0}\times \textit{Wander}_K(\phi)\times R_S^* : \phi^{(n+k)}(f)^r=u\phi^{(k)}(f)^s \right\}.  $$ 
If $ (n,k,f,u) \in \mathcal{E}_\rho(K,\phi, S, r,s)$, then $n$ and $k$ are bounded effectively in terms of $\phi,K, \rho,|S|$ and $\inf_{\hat{h}_\varphi(P)>0}\hat{h}_\varphi(P)$, and $f$ and $u$ have height bounded from above effectively in terms of $\phi,\rho,|S|$ and $K$.
\end{theorem}
\begin{proof}
Arguing in the same way as in the proof of \cite[Theorem 1.3]{BOSS}, we can assume that $0$ is not an exceptional point for $\phi$, which allows us to apply Theorem \ref{th2.9}(b) with $A=0$, and $\epsilon$ to be specified later, yielding an effective constant $\gamma_3(\phi,\epsilon,|S|,K)$ such that
$$
\max \left\{ n \geq 0 : \displaystyle\sum_{v \in S} \log^+(|\phi^n(\alpha)|_v^{-1})\geq \epsilon \hat{h}_\phi(\phi^n(\alpha))\right\} \leq \gamma_3
$$

 We study triples $(n,k,\alpha) \in \mathbb{Z}_{\geq 1}\times \mathbb{Z}_{\geq 0}\times \text{Wander}_K(\phi)$ such that 
$$
|\phi^{n+k}(\alpha)|_v^r=|\phi^{k}(\alpha)|_v^s \quad \text{ for all } v \in M_K\setminus S.
$$ We then divide the proof in two cases.
\newline \textbf{Case 1: $n+k \geq \gamma_3$.}\newline
In this case Theorem \ref{th2.9} tells us that $(n,k,\alpha)$ satisfies
\begin{equation}
 \displaystyle\sum_{v \in S} \log^+(|\phi^{n+k}(\alpha)|_v^{-1})\leq \epsilon \hat{h}_\phi(\phi^{n+k}(\alpha)).
\end{equation}we compute

\begin{align*}
h(\phi^{n+k}(\alpha)) &=h(\phi^{n+k}(\alpha)^{-1})= \sum_{v\in S} \log^+ |\phi^{n+k}(\alpha)|^{-1}_v+\sum_{v\notin S} \log^+ |\phi^{n+k}(\alpha)|^{-1}_v\\
& \leq   \epsilon \hat{h}_\phi(\phi^{n+k}(\alpha))\  +\sum_{v\notin S} \log^+ |\phi^{k}(\alpha)|^{-s/r}_v\\
& \leq   \epsilon \hat{h}_\phi(\phi^{n+k}(\alpha))\  +\left|\frac{s}{r}\right|\sum_{v\notin S}h(\phi^{k}(\alpha)^{-1})\\
& \leq  \epsilon \hat{h}_\phi(\phi^{n+k}(\alpha))\  +d^{\rho-1}h(\phi^k(\alpha))\\
& \leq \epsilon \hat{h}_{\phi}(\phi^{n+k}(\alpha)) + d^{\rho-1} (\hat{h}_{\phi}(\phi^k(\alpha)) + O(h(\phi)+1))\\
\end{align*} implying that
$$
(1-\epsilon)d^{n+k} \hat{h}_\phi(\alpha)
\leq  d^{\rho-1+k}\hat{h}_\phi(\alpha)+ O(d^{\rho-1}(h(\phi))+1),$$ which with $\epsilon =1/3$ yields that
$$
d^{n+k-1}\hat{h}_\phi(\alpha)\leq O(d^{\rho-1}(h(\phi))+1),
$$
 and hence
$$
n,k \leq \gamma_4 +  \log^+_{d} \left( \dfrac{h(\phi)}{\inf_{\hat{h}_\phi(P)>0}\hat{h}_{\phi}(P)} \right)
$$ for an effective constant $\gamma_4(K,|S|,\phi,\rho)$. Given that there are at most finitely many possibilities for $n$ and $k$, this also yields a desired effective height bound for the referred $\alpha$'s. 
\newline \textbf{Case 2: $n+k \leq \gamma_3$.}\newline
Since$n$ and $k$ are bounded, we may assume that they are fixed. In this case, we let
$$ g(z)=\phi^n(z)^r/z^s$$ so that we have $$g(\phi^k(\alpha))=\phi^{n+k}(\alpha)^r/\phi^k(\alpha)^s \in R_S^*. $$ This says that $$\phi^k(\alpha) \in \{ f \in K : g(f) \in R^*_S\}$$
By \ref{lem3.1}, the set above has cardinality effectively bounded, and its elements as well as the referred $\alpha$'s have height effectively bounded from above as desired, except if $g$ has at most two zeros and poles. But the assumption that $0$ is not periodic implies that $0$ is a pole of $g$, and the assumption that $\phi(X) \neq cX^{\pm d}$  implies that $\phi^n$ has at least two poles or zeros distinct from $0$. Hence $g$ has at least three poles and zeros.
\end{proof}\subsection{Polynomial dynamics and multiplicative dependence over function fields in one variable.}
In this subsection, we consider $K$ to be an algebraic function field in one variable of characteristic zero. When we deal with polynomials in this setting, one can obtain independence on the exponents $r,s$ from Theorem \ref{th3.3}.

The result stated below is analogous to a Schinzel-Tijdeman result, now over the fields treated here.
\begin{lemma}\cite{BPV}\label{lem3.4}
Let $f \in K[X]$ be a polynomial of degree $n$ with at least two distinct zeros in some algebraic closure of $K$. Then the equation $$f(x)=y^m \text{ in } x,y \in K, m \in \mathbb{Z}_{\geq 0}, y \notin k$$ implies 
$$m \leq B(f,K),$$ where $B=B(f,K)$ is an effective constant.
\end{lemma}The next result generalizes \cite[Theorem 1.7]{BOSS} to function fields in one variable.
\begin{theorem}
Let $\phi \in K[z]$ be a non-isotrivial polynomial of degree $d\geq 2$ such that $\phi$ and $\phi^2$ have no multiple roots, and $0$ is not a periodic point of $\phi$. Let $S \in M_K$ be a finite set of places. Then there exist a set $T_1$ with cardinality effectively bounded in terms of $K,\phi$ and $S$, and a set $T_2$ whose elements have height effectively bounded from above in terms of $K,\phi$ and $S$, such that$$
\{f \in K: (\phi^m(f))^r=u(\phi^n(f))^s, u \in R^*_S, m>n\geq0, r,s \in \mathbb{Z}\}=T_1 \cup T_2.
$$
\end{theorem}
\begin{proof}
Let $\alpha \in \text{Wander}_\phi(K)$ such that there exist non-negative integers $m>n>0$, integers $r$ and $s$, and $u \in R^*_S$ such that \begin{equation}\label{eq3.2}(\phi^m(\alpha))^r=u(\phi^n(\alpha))^s.\end{equation}
If $r=0$ or $s=0$, the result follows from Theorem \ref{th3.2}, so we may assume that $rs \neq 0.$ Also, due to the saturation $R^*_S$, we may assume that gcd$(r,s)=1$. We enlarge $S$ by the places for which $\phi$ has bad reduction, namely, those $v \in M_K$ such that if $f(z)=c_0+c_1z+...+c_dz^d$ then $v(c_i)<0$ for some $i$ or $v(c_d)>0$, obtaining a new finite set of places $S_\phi$. In fact, proving the results for the larger $S_f$ gives results that are stronger than the original statements. In this case, one can see that $$S_{\phi^m}\subset S_\phi \text{ for all }m \geq 1.$$ Since $rs \neq 0,$  replacing $r,s$ by $-r,-s$ if necessary, we may assume that $r>0$. We now divide the proof in  cases.\newline
\textbf{Case A: $\alpha \in R_{S_\phi}$.}\newline
By definition one can check that $\phi^k(\alpha) \in R_{S_\phi}$ for every $k\geq 0$, so that
$$v(\phi^k(\alpha))\geq 0 \text{ for all } k \geq 0, v \notin S_\phi.$$ We now divide this step in some subcases.\newline
\textbf{Case A.1: $r>0, s<0$.}\newline Here, equation \ref{eq3.2} becomes
$$(\phi^m(\alpha))^r(\phi^n(\alpha))^t=u \text{ with } t=-s>0.  $$
As $u \in R^*_{S_\phi},$ it follows that
$$
rv(\phi^m(\alpha))+tv((\phi^n(\alpha))=0 \text{ for all } v \notin S_\phi.
$$Since $r,t>0$, this implies that 
$$
v(\phi^m(\alpha))=v((\phi^n(\alpha))=0 \text{ for all } v \notin S_\phi,
$$ and hence that $\phi^m(\alpha) \in R^*_S$. The conclusions come then from Theorem \ref{th3.2}.\newline
\textbf{Case A.2: $r>0, s\geq 2$.}\newline Since gcd$(r,s)$, we can choose $a$ and $b$ with $ar+bs=1$ so that \ref{eq3.2} becomes $$\phi^m(\alpha)=u^a\left((\phi^n(\alpha))^a(\phi^m(\alpha))^b  \right)^s. $$Since $\phi^m(\alpha) \in R_{S_\phi}$ and $u \in R^*_{S_\phi}$, we have that $(\phi^n(\alpha))^a(\phi^m(\alpha))^b \in R_{S_\phi}$. If $(\phi^n(\alpha))^a(\phi^m(\alpha))^b \in R_{S_\phi}^*$, then $\phi^m(\alpha) \in R_{S_\phi}^*$ and we have the desired results by Theorem \ref{th3.2}. If not, then we also have that $(\phi^n(\alpha))^a(\phi^m(\alpha))^b \notin k$. Writing $$\phi^m(\alpha)=\phi(\phi^{m-1}(\alpha)),$$ we use Lemma \ref{lem3.4} to conclude that the exponent $s\geq 2$ is effectively bounded in terms of $\phi$ and $K$.

If we first assume that $\deg \phi \geq 3$, then we can apply \cite[Proposition 4.6]{BEG0} to effectively bound the referred $\phi^{m-1}(\alpha)$'s, and then $m$ and $h(\alpha)$ as desired after using the similar method to the one used in the proof of Theorem \ref{th3.2}(b). 

If otherwise $\deg \phi=2$, we can apply  \cite[Proposition 4.7]{BEG0}, and conclude similarly.\newline
\textbf{Case A.3: $r\geq 2, s=1$.}\newline
If $n \geq 2$, then the same discussion as above holds again (replacing $m$ by $n$ and $r$ by $s$). It is enough then to consider the case $n=1$, so that \ref{eq3.2} becomes $$\phi(\alpha)=u^{-1}(\phi^m(\alpha))^r.$$ If $r \geq 3$, then we can apply \cite[Proposition 4.7]{BEG0} again to conclude.
If otherwise $r=2$, we can apply Theorem \ref{th3.3} to conclude as desired.\newline
\textbf{Case A.4: $r=1, s=1$.}\newline This case is covered by Theorem \ref{th3.3}.\newline
\textbf{Case B: $\alpha \notin R_{S_\phi}$.}\newline
Choosing $v \in M_K\setminus S_\phi$ such that $v(\alpha)<0$, we have also that $v \notin S_{\phi^k}$ for all $k\geq1$. In this case, we have that $v(\phi^k(\alpha))=d^kv(\alpha)$ for all $k\geq 0$ and hence \ref{eq3.2} implies
$$
rd^mv(\alpha)=sd^nv(\alpha).
$$ Therefore $r=1$ and $s=d^{m-n}$, and \ref{eq3.2} becomes
$$
\phi^{n+k}(\alpha)=u(\phi^k(\alpha))^{d^{n}}.
$$If $n=1$, $\phi^{1+k}(\alpha)=u(\phi^k(\alpha))^{d^{n}},$ and we can use Theorem \ref{th3.2}(a) applied to the rational function $f(X)/X^d$ to conclude the desired. 

Otherwise, we first enlarge $S_\phi$ in a controlled way such that $\text{disc}(\phi), \text{disc}(\phi^2) \in R^*_{S_\phi}$. Moreover, we note that the group $R^*_{S_\phi}/(R^*_{S_\phi})^{d^2}$ is finite and effectively bounded due to \cite[III, Prop.2.3.2]{Silv0}, allowing us to replace $K$ by the finite extension $L$ generated by the set of values $$
\{\beta \in \overline{K} : \beta^{d^2} \in R^*_{S_\phi} \},
$$ which depends only on $K,d$ and $S$.

If $d\geq 3$ or $d=2$, we consider  respectively the curves $$\phi(X)=Y^{d^2} \text{ or }\phi^2(X)=Y^{d^2}.$$
In both cases, we are able to apply \cite[Theorem 2]{ES} to obtain a set $T_1$ as in the statement with the desired finiteness properties.\end{proof}
\subsection{Iterates as zeros of split polynomials} Here we would like to point that Theorem 1.10 of \cite{BOSS} also is true over function fields of characteristic $0$, making use of Lemma \ref{lem2.5}. We start recalling the following definition
\begin{definition}
 We define a multilinear polynomial with split variables to be a vector of polynomials
 $$
F(\textbf{T}_1,...,\textbf{T}_k)=\sum_{i=1}^r c_i \prod_{j \in J_i} \textbf{T}_j \in  K[\textbf{T}_1,...,\textbf{T}_k]
 $$ for some disjoint partition $J_1 \cup ...\cup J_r=\{1,...,s\}$ and $c_i \in K^*, i=1,...,r$.
 \end{definition}
 
  \begin{theorem}  Let $K$ be a function field of a smooth projective curve over an algebraically closed field of characteristic $0$.
 Let $F(\textbf{T}_1,...,\textbf{T}_k) \in K[\textbf{T}_1,...,\textbf{T}_k]$ be a multilinear polynomial with split variables and let $\phi \in K(z)$ be a non-isotrivial rational function degree $d \geq 2$.
 
   The set of  $\alpha \in \overline{K}$ not preperiodic for $\phi$, for which there exists a $k$-tuple of distinct non-negative integers $(n_1,...,n_k)$  satisfying $$F(\phi^{(n_1)}(\alpha),...,\phi^{(n_k)}(\alpha))=0  $$ is a set of bounded height. If $d\geq 3$, such heights are bounded effectively in terms of $\phi, F$ and $K$.
   
   Moreover, there are only finitely many $k$-tuples of integers $n_1>n_2>...>n_k$ satisfying $F(\phi^{(n_1)}(\alpha),...,\phi^{(n_k)}(\alpha))=0$, and there is a bound for such integers that depend only on $\phi, F, K$ and $\inf_{\hat{h}_\phi(P)>0} \hat{h}_\phi(P)$, independent on $\alpha$.
\end{theorem}
\begin{proof}
The proof follows the number field situation \cite[Theorem 1.10]{BOSS} almost ipsis literis, except that we have to point that the quanitity $C_2(K,f)$ in that proof is replaced here by $\inf_{\hat{h}_\phi(P)>0} \hat{h}_\phi(P)$, which we know to exist due to Lemma \ref{lem2.5}.
\end{proof}

\textbf{Acknowledgements}: I am thankful to Australian Research Council Discovery Grant DP180100201  and UNSW for supporting me in this research.

\end{document}